\documentclass[12pt]{amsart}

\usepackage{amssymb,bm,mathrsfs}

\usepackage{enumerate}
\usepackage{tikz}
\usepackage[centering,width=6in]{geometry}
\usepackage[colorlinks,pagebackref,pdfstartview=FitH]{hyperref}

\theoremstyle{plain}
\newtheorem{thm}{Theorem}[section]
\newtheorem{prop}[thm]{Proposition}
\newtheorem{lem}[thm]{Lemma}

\theoremstyle{definition}

\theoremstyle{remark}

\numberwithin{equation}{section}

\newcommand{\N}{\mathbb N}

\newcommand{\R}{\mathbb R}

\title{Erd\H{o}s similarity problem via bi-Lipschitz embedding}

\author{De-jun Feng}
\address{
Department of Mathematics\\
The Chinese University of Hong Kong\\
Shatin,  Hong Kong
}
\email{djfeng@math.cuhk.edu.hk}
\author{Chun-Kit Lai}
\address{Department of Mathematics\\
San Francisco State University\\
1600 Holloway Avenue, San Francisco, CA 94132
}
\email{cklai@sfsu.edu}

\author{Ying Xiong}
\address{
Department of Mathematics\\
South China  University of Technology\\
Guangzhou 510641,  Guangdong\\
People's Republic of China
}
\email{xiongyng@gmail.com}

\begin{document}

\begin{abstract}
 The Erd\H{o}s similarity conjecture asserted that an infinite set of real numbers cannot be affinely embedded into every measurable set of positive Lebesgue measure. The problem is still open, in particular for all fast decaying sequences.  In this paper, we relax the problem to the bi-Lipschitz embedding and obtain some sharp criteria about the bi-Lipschitz Erd\H{o}s similarity problem for strictly decreasing sequences.
\end{abstract}

\maketitle

\section{Introduction and Main Results}

\subsection{Background} Searching for ``copies" of  patterns of certain points inside sets of fairly big  size  has been a central problem in different branches of mathematics. Such a statement can take many different forms depending on what ``copies" people are looking for and the notion of the ``size". The most natural notion in this regard would be an affine copy.  By an {\bf affine copy} of a set $P\subset \R$, we mean a set of the form $\lambda P+ t$, where $\lambda, t\in \R$ and $\lambda\neq 0$.


\medskip

The notion of ``size" of sets takes many different forms as well, but one of the most natural choices is undoubtedly the Lebesgue measure. The earliest result about the abundance of patterns for sets of positive measure was due to Steinhaus \cite{Steinhaus}. Using the Lebesgue density theorem, Steinhaus proved that if $P$ is a finite set of real numbers, then every measurable set $E\subset \R$ of positive Lebesgue measure contains an affine copy of $P$.  Here and afterwards ``measurable''  always means ``Lebesgue measurable''.  Erd\H{o}s however believed that infinite sets cannot be that abundant. Following the notion introduced by Kolountzakis \cite{Kol97}, we say that a set $P$ is {\bf measure universal} if every measurable set of positive Lebesgue measure contains an affine copy of $P$.  In early 1970's,  Erd\H{o}s  (see example \cite{Erd74} or \cite{Erd15}) proposed a conjecture, which is open until today.

\medskip

{\bf Erd\H{o}s similarity Conjecture:} There is no infinite measure universal set.

\medskip

To validate that the conjecture is true, it would suffice to show that for any positive strictly decreasing sequence $(a_n)_{n=1}^\infty$ with $a_n\to 0$, there exists a Lebesgue measurable set $E$  of positive measure such that $E$ contains no affine copy of this sequence. The first progress was made by  Eigen \cite{Eig85} and Falconer \cite{Fal84}  independently who showed that if  $(a_n)_{n=1}^\infty$ is a decreasing sequence converging to $0$ slowly in the sense that
\begin{equation}
\label{e-e0}
\lim_{n\to\infty} \frac{a_{n+1}}{a_n} = 1,
\end{equation}
then $(a_n)_{n=1}^\infty$ is not measure universal. Humke and Laczkovich \cite{HL98} further showed that the condition \eqref{e-e0} can be weaken to $
\limsup_{n\to\infty} {a_{n+1}}/{a_n} = 1$ if in addition $(a_n-a_{n+1})$ is monotone decreasing.    Kolountzakis \cite{Kol97} showed that  an infinite set $A\subset \R$    is not measure universal if for all $n\in\N$ sufficiently large, $A$ contains a subset  $\{a_1>a_2>\cdots>a_n>0\}$ such that
$$
-\log \left( \min_{1\leq i\leq n-1}\frac{a_i-a_{i+1}}{a_1}\right) = o(n).
$$
His result can be used to recover  that of Eigen and Falconer. Using a probabilistic argument,  Kolountzakis further showed that the conjecture is   almost surely true: for every sequence $A= (a_n)_{n=1}^{\infty}$, there exists a  measurable set $E$ of positive Lebesgue measure such that  the set
$$
\{(\lambda,t)\in \R^2: \lambda A+t\subset E \}
$$
has two dimensional Lebesgue measure zero. It is still an open question to determine whether any given sequence $(a_n)_{n=1}^\infty$,  whose ratio limit is  strictly less than 1 (e.g. $a_n=2^{-n}$), is measure universal. The reader is referred to \cite{CLP22} for some recent progresses and to \cite{Svetic} for other progresses of the problem before year 2000.

\medskip

 There are  some other types of infinite sets determined to be measure non-universal via different methods. Bourgain \cite{Bou87} converted the conjecture into a problem about  the boundedness of certain operators (see also \cite{Tao21}). He showed that $S_1+S_2+S_3$ is not measure universal if all $S_i$ are infinite sets.  It is not known if this result is also true about the sum of two infinite sets. Along this direction, Kolountzakis \cite{Kol97} showed a special case that $
 \{2^{-n^{\alpha}}\}_{n=1}^{\infty}+ \{2^{-n^{\alpha}}\}_{n=1}^{\infty}
 $ with $\alpha\in(0,2)$
 is not measure universal. More recently, Gallagher, the second named author and Weber \cite{GLW23} showed that Cantor sets of positive Newhouse thickness are not measure universal. This was followed up immediately by Kolountzakis \cite{Kol23} who showed that some Cantor sets of Newhouse thickness zero was also measure non-universal. Nonetheless, despite being uncountable, it is still unknown  if all Cantor sets are not measure universal.

 \medskip

 Another interesting variant, known as  the Erd\H{o}s similarity problem ``in the large",  was first initiated by Bradford, Kohut and Mooroogen \cite{BKM22} and later improved in \cite{KP23}. In these work, they showed that for any  unbounded sequence of certain restricted increasing rate and $p\in(0,1)$, there always exists  a  measurable set $E$ of positive Lebesgue measure  such that ${\mathcal L}(E\cap [x,x+1])\ge p$ for all $x\in\R$ and $E$ does not contain an affine copy of this sequence. Here ${\mathcal L}$ stands for Lebesgue measure. In a further improvement, Burgin-Goldberg-Keleti-MacMahon-Wang \cite{BGKMW22} turned a set ``in the large'' back to  a compact set  $E\subset \R$ such that $0$ is a Lebesgue density point of $E$
but $E$ does not contain any (non-constant) infinite  geometric progression.

\medskip

It is worth to mention that  the Erd\H{o}s similarity problem has arose many different interests and variations of the study under other notions of size, such as  Hausdorff dimension.  For any countable collection of sets of three points, Keleti \cite{Ke2008} constructed a compact subset of the real line with Hausdorff dimension $1$
 that contains no affine copy of any of the given triplets.  On the other hand, there also exists a closed set  of real numbers with Hausdorff dimension zero which contains affine copies of all finite sets \cite{DMT60} (such a set must have packing dimension 1, see  \cite[Lemma 5.3]{SS2017}).   By assuming certain Fourier decay conditions, {\L}aba and Pramanik \cite{LP2009} showed that a large class of fractal sets contain  non-trivial 3-term arithmetic progressions.   Meanwhile Shmerkin constructed some Salem sets which contain no 3-term arithmetic progression \cite{S2017}. A recent in-depth study of the detection of patterns in relation to the Fourier dimension can be found in \cite{LP22}.

\subsection{Main Results}

In a private communication \cite{Jin23}, Xiong Jin proposed another variant of the Erd\H{o}s similarity problem by considering the bi-Lipschitz copies, instead of the affine copies.  Indeed, bi-Lipschitz or $C^1$ embedding problems have been studied intensively between self-similar sets (see e.g. \cite{DWXX11,FHR15, Alg20}).  In this paper we will prove a sharp characterization to  decreasing sequences being universal in the sense of having bi-Lipschitz copies.

\medskip

 Recall that a  map $f:\R\to\R$ is  said to be {\bf bi-Lipschitz} if there exists a constant $L>1$ such that
$$
L^{-1} |x-y| \le |f(x)-f(y)|\le L |x-y|\quad \mbox{ for all }\; x,y\in\R.
$$
A {\bf bi-Lipschitz copy} of a set $A$ is the image $f(A)$ where $f$ is a bi-Lipschitz map.  Clearly, an affine copy must be a bi-Lipschitz copy. If a bi-Lipschitz copy of $A$ is contained  in a set $E$, we will say that $A$ can be {\bf bi-Lipschitz embedded} into $E$. We will say that $A$ is {\bf  bi-Lipschitz measure universal} if $A$ can be bi-Lipschitz embedded into every measurable set of positive Lebesgue measures.  Our first  result is the following.

\begin{thm}\label{t:BLE}
    Let $(a_n)_{n=1}^\infty$ be a strictly decreasing sequence of positive numbers with $a_n\to0$ as $n\to\infty$. If there exists an integer $N\ge1$ such that
	\[ \limsup_{n\to\infty} \frac{a_{n+N}}{a_n}<1, \]
	then for any measurable set $E\subset\R$ with positive Lebesgue measure, there exists a bi-Lipschitz map $f\colon\R\to\R$ such that $f(a_n)\in E$ for all $n\ge1$ and $f'(0)=1$.
\end{thm}

This result shows that fast decaying sequences like $(2^{-n})_{n=1}^\infty$ is bi-Lipschitz measure universal.  Moreover, the derivative condition implies that the map becomes very close to an affine map at the limit point. This suggested some negative evidence of the Erd\H{o}s similarity conjecture, indicating that the set avoiding affine copies of fast decaying sequences will not be so easily constructed. Furthermore, as our proof will be based on the Lebesgue density theorem,  in comparison to the aforementioned result of \cite{BGKMW22}, we can still deduce the existence of  bi-Lipschitz copies of geometric sequences around every density point.

\medskip

In contrast to  the above result for fast decaying sequences, our second result states that a slow decaying sequence can not be bi-Lipschitz embedded into all measurable sets of positive Lebesgue measure,  which generalizes the aforementioned affine embedding result by Eigen \cite{Eig85} and Falconer \cite{Fal84}. This also provides a new proof of their result.

\begin{thm}\label{p:nonemb}
    Let $(a_n)_{n=1}^\infty$ be a strictly decreasing sequence of positive numbers with $a_n\to0$ as $n\to\infty$. If
	\[ \lim_{n\to\infty}\frac{a_{n+1}}{a_n}=1,\]
	then there exists a  compact set $E\subset\R$ with positive Lebesgue measure such that $(a_n)_{n=1}^\infty$ can not be bi-Lipschitz embedded into~$E$.
\end{thm}

More generally, we provide a classification theorem for a type of decreasing sequences. The condition includes the commonly-known convex deceasing sequences i.e. $a_{n+1}\le \frac{a_{n}+a_{n+2}}{2}$ for all $n\ge 2$.

\begin{thm}\label{t:iff}
	Let $(a_n)_{n=1}^\infty$ be a strictly decreasing sequence of positive numbers with $a_n\to0$ as $n\to\infty$. Suppose in addition
	\begin{equation}
	\label{e-eT}
	 \sup_{m>n>1}\frac{a_{m-1}-a_m}{a_{n-1}-a_n}<\infty.
	\end{equation}
	Then $(a_n)_{n=1}^\infty$ is bi-Lipschitz measure universal  if and only if
	$\limsup_{n\to\infty}\frac{a_{n+1}}{a_n}<1. $
\end{thm}

The `only  if' part of the above theorem generalizes the aforementioned affine embedding result of Humke and Laczkovich \cite{HL98}.  It is worth pointing out that  the condition \eqref{e-eT}  in Theorem \ref{t:iff} can not be dropped.  Indeed, there  are many examples of  decreasing sequences $(a_n)_{n=1}^\infty$
with $\lim_{n\to \infty} a_n=0$ such that
$$
\limsup_{n\to \infty} \frac{a_{n+2}}{a_n}<1\quad  \mbox{ but } \quad \limsup_{n\to \infty} \frac{a_{n+1}}{a_n}=1.
$$
For instance, this is the case if
\[ a_n= \begin{cases}
	(2^k-1)^{-1}, & \text{if $n=2k-1$}, \\
	2^{-k}, & \text{if $n=2k$}.
\end{cases} \]
By Theorem \ref{t:BLE},  the above sequence $(a_n)_{n=1}^\infty$ can be  bi-Lipscitz embedded into every measurable set of positive Lebesgue measure, although $\limsup_{n\to \infty}{a_{n+1}}/{a_n}=1$.

\medskip

  Apart from fast decaying decreasing sequence being bi-Lipschitz embedded into every measurable set of Lebesgue measures as shown in Theorem \ref{t:BLE}, we finally demonstrate a type of countable sets with infinitely many limit points are also bi-Lipschitz measure universal.

\begin{thm}\label{th:uniform}
	Let $A = (a_n)_{ n=1}^{\infty}$ be a sequence of positive numbers such that
	\[ a_1+\sum_{n=1}^{\infty}\frac{a_{n+1}}{a_n}<\frac{1}{8}. \]
	Then the set
	\[ F = \bigcup_{n=1}^{\infty} 3^{-n} (1+A) \]
	is bi-Lipschitz  measure universal.
\end{thm}

Theorem \ref{th:uniform}  has  demonstrated that a set with infinitely many limit points  could be universally bi-Lipschitz embedded. It will be an interesting question to determine  whether or not sets that were previously studied, such as  a Minkowski sum of three infinite sets or even a Cantor set, can be bi-Lipschitz measure universal.   For Cantor sets, we actually know that if a bi-Lipschitz measure universal Cantor set exists, it must have Newhouse thickness zero. It is because in \cite[Theorem 1.5]{GLW23}, it has been shown that there exists a set $G$ of full Lebesgue measure not containing any Cantor sets of positive Newhouse thickness. As a bi-Lipschitz image of a Cantor set with positive Newhouse thickness still has positive Newhouse thickness, the existence of the set $G$ immediately implies that Cantor sets with positive Newhouse thickness can not be bi-Lipschitz measure universal. 

Moreover, it is worth noting that Theorem~\ref{th:uniform} also implies that the condition in Theorem~\ref{t:BLE} is not necessary to ensure the bi-Lipschitz embedding. To see this, let $A = ( 4^{-4^n})_{n=1}^{\infty}$. One can check that $A$ satisfies the assumption of Theorem~\ref{th:uniform}. Now let $A_n$ be a finite subset of $A$ with cardinality $\ge n$ for each $n\ge1$. Set $E=\bigcup_{n=1}^\infty  3^{-n} (1+A_n).$ Then $E$ is a decreasing sequence such that $\limsup_{n\to\infty} a_{n+N}/a_n=1$ for all $N\in\N$ \footnote{$\liminf_{n\to\infty} a_{n+1}/a_n<1$ still holds, so it does not contradict Theorem \ref{p:nonemb}.}, yet $E$ is still bi-Lipschitz measure universal.  Hence, the converse of Theorem \ref{t:BLE} is also not true.

The proof of Theorem~\ref{th:uniform} involves a study about universally bi-Lipschitz embedding with uniform bi-Lipschitz bound. A quantitative result will be given in Theorem \ref{th:uniform-biLip}.


The paper is organized as follows. In Section~\ref{S-2}, we prove Theorem \ref{t:BLE}. In Section~\ref{S-3}, we prove Theorems~\ref{p:nonemb} and \ref{t:iff}. We will prove Theorem \ref{th:uniform} with a study of uniform  universal bi-Lipschitz embedding in Section \ref{S-4}.

\section{ The proof of Theorem \ref{t:BLE}}
\label{S-2}

Let $\mathcal L$ denote the Lebesgue measure on $\R$. The proof of Theorem \ref{t:BLE} is based on the following.

\begin{lem}\label{l:Ik} Let $N\in \N$ and $\delta\in (0,1)$. Let $E$ be a Lebesgue measurable subset of $\R$ and $I=[\delta u, v]$, where
\begin{equation}
\label{e-e8}
0<\delta^{N-1}v\leq u\leq v.
\end{equation}
Suppose that
\begin{equation}
\label{e-e9}
\rho:=\frac{\mathcal L(I\setminus E)}{\mathcal L(I)}<N^{-2}\delta^N.
\end{equation}
 Then for any $x_1, \ldots, x_N\in [u,v]$, there exists $0\leq t\leq (1-\delta)N^2\delta^{-N}\rho u$ such that $x_1-t,\ldots, x_N-t\in E\cap I$.
\end{lem}

\begin{proof}
Let $x_1,\ldots, x_N\in [u,v]$. By \eqref{e-e9},  $N^2\delta^{-N}\rho<1$.  Hence  $$x_1-t,\ldots, x_N-t\in [\delta u, v]=I$$
for every $t\in \left[0, (1-\delta)N^2\delta^{-N}\rho u\right]$.
Below we show by contradiction that there always exists $t\in \left[0, (1-\delta)N^2\delta^{-N}\rho u\right]$ such that $x_i-t\in E$ for all $1\leq i\leq N$.

Suppose on the contrary that the above conclusion is false. Then for each $t\in \left[0, (1-\delta)N^2\delta^{-N}\rho u\right]$, there exists $i$ such that $x_i-t\in I\setminus E$, or equivalently, $t\in x_i-(I\setminus E)$. Hence
$$
\left[0, (1-\delta)N^2\delta^{-N}\rho u\right]\subset \bigcup_{i=1}^N \left(x_i-(I\setminus E)\right).
$$
It follows that
\begin{equation}
\label{e-e7'}
(1-\delta)N^2\delta^{-N}\rho u\leq N \mathcal L(I\setminus E).
\end{equation}
However,
 \begin{align*}
  \mathcal L(I\setminus E)&=\mathcal L(I)\rho \qquad\qquad\qquad\quad \mbox{(by \eqref{e-e9})}\\
  &=(v-\delta u)\rho  \\
  &\leq \left(\delta^{-(N-1)}-\delta\right) \rho u \qquad\; \;\mbox{(by \eqref{e-e8})} \\
  &= (1-\delta)\left(\delta+\cdots+\delta^{N}\right) \delta^{-N}\rho u\\
  &<(1-\delta) N\delta^{-N}\rho u,
   \end{align*}
leading to a contradiction with  \eqref{e-e7'}.
\end{proof}

Now we are ready to prove Theorem \ref{t:BLE}.

\begin{proof}[Proof of Theorem \ref{t:BLE}]

Since $(a_n)_{n=1}^\infty$ is strictly decreasing and $\limsup_{n\to\infty} \frac{a_{n+N}}{a_n}<1$, there exists $\delta\in (0,1)$ such that
\begin{equation}\label{eq:eps}
	\frac{a_{n+N}}{a_n}<\delta^N \quad\mbox{ for all }\;n\ge1.
\end{equation}
Let $(n_k)_{k=0}^\infty$ be the increasing sequence of positive integers given by $n_0=1$ and
\begin{equation}\label{eq:nk}
	\{n_k\colon k\ge1\}=\bigl\{ n\geq 1\colon a_{n+1}/a_n<\delta \bigr\}.
\end{equation}
We claim that
\begin{equation}\label{eq:nk-}
	n_{k+1}-n_k\le N \quad\text{for all $k\ge1$}.
\end{equation}
Otherwise if  $n_{k+1}-n_k>N$ for some $k\ge1$, then by \eqref{eq:nk}, $$a_{j+1}/a_j\geq \delta$$ for $j=n_k+1,\ldots, n_{k}+N$, implying that
\[
	\frac{a_{n_k+1+N}}{a_{n_k+1}} =\prod_{j=n_k+1}^{n_k+N} \frac{a_{j+1}}{a_{j}}\geq \delta^N,
 \]
which contradicts~\eqref{eq:eps}. This proves \eqref{eq:nk-}.

To simplify the notation, for $k\geq 1$ we write
\begin{equation}\label{eq:Ik}
	u_k=a_{n_k},\quad v_k=a_{n_{k-1}+1},\quad \mbox{and}\quad I_k=[ \delta u_k,v_k].
\end{equation}
By~\eqref{eq:nk} and~\eqref{eq:nk-},
\begin{equation}\label{eq:betak1}
	\frac{v_{k+1}}{u_k} =\frac{a_{n_k+1}}{a_{n_k}}<\delta,
\end{equation}
and \begin{equation}
\label{eq:betak}
	1\geq \frac{u_k}{v_k}=\frac{a_{n_k}}{a_{n_{k-1}+1}}= \prod_{j=n_{k-1}+1}^{n_k-1}\frac{a_{j+1}}{a_j} \ge \delta^{n_k-n_{k-1}-1}\ge \delta^{N-1}.
\end{equation}
It follows from \eqref{eq:betak1} that $I_k\cap I_{k+1}= \varnothing$. As $(v_k)_{k=1}^\infty$ is monotone decreasing,  the intervals $I_k$ are disjoint.

Let $E\subset\R$ be measurable with positive Lebesgue measure. Replacing $E$ by its suitable translation if necessary, we may assume that $0$ is a Lebesgue density point of~$E$, that is,
\[ \lim_{r\to 0}\frac{\mathcal L([0,r]\cap E)}{r}=1. \]
 Then $\lim_{r\to 0}{\mathcal L( [0,r]\setminus E)}/{r}=0$. Since  $\lim_{k\to \infty} v_k=0$, it follows that
 $$\frac{\mathcal L(I_k\setminus E)}{\mathcal L(I_k)}\leq \frac{\mathcal L([0, v_k]\setminus E)}{\mathcal L(I_k)}= \frac{\mathcal L([0, v_k]\setminus E)}{v_k}\cdot \frac{v_k}{v_k-\delta u_k}\leq \frac{\mathcal L([0, v_k]\setminus E)}{v_k}\cdot \frac{1}{1-\delta}\to 0$$
 as $k\to \infty$. Write
  \begin{equation}\label{eq:rhok}
  \rho_k:=\frac{\mathcal L( I_k\setminus E)}{\mathcal L(I_k)},\quad k\geq 1.
  \end{equation}
Choose a large integer $p$ such that
\begin{equation}
\label{e-e10}
 \rho_k<N^{-2}\delta^N \quad \mbox{ for all }\;k\geq  p.
\end{equation}
Since $0$ is a density point of $E$ we may also assume that $\mathcal L(E\cap (v_p,\infty))>0$.

Next we construct a strictly decreasing sequence $(b_n)_{n=1}^\infty$ of positive numbers such that $b_n\in E$ for $n\geq 1$, $\lim_{n\to \infty} b_n=0$, and moreover,
\begin{equation}
\label{e-e11}
 \lim_{n\to \infty}\frac{b_n-b_{n+1}}{a_n-a_{n+1}}=1.
\end{equation}
To this end, we first arbitrarily choose a strictly decreasing sequence $(b_j)_{j=1}^{n_p}$ of positive numbers  from $E\cap (v_p,\infty)$.  Then for each $k\geq p+1$,
by \eqref{eq:betak}, \eqref{e-e10} and Lemma \ref{l:Ik} (in which we take $I=I_k$),  we can find $0\leq t_k\leq (1-\delta)N^2\delta^{-N}\rho_k u_k$ such that
$$
a_j-t_k\in E\cap I_k \quad \mbox{ for  all }\;  n_{k-1} +1\leq j\leq n_{k}.
$$
Define $b_j=a_j-t_k$ for $ n_{k-1} +1\leq j\leq n_{k}$. In this way, we obtain the sequence $(b_n)_{n=1}^\infty$.  See Figure \ref{fig:a-t} for an illustration of our construction.

\medskip

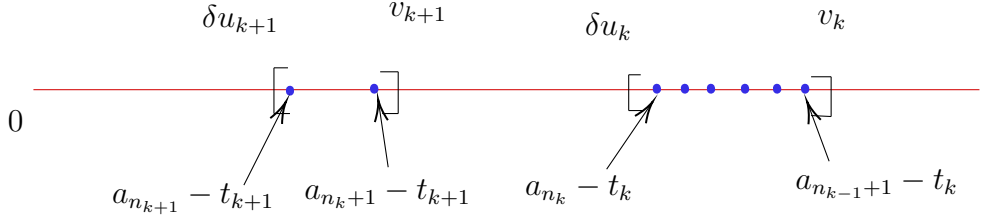
\begin{figure}[h]
\begin{tikzpicture}[x=0.75pt,y=0.75pt,yscale=-1,xscale=1]

\draw [color={rgb, 255:red, 212; green, 10; blue, 10 }  ,draw opacity=1 ]   (117,114) -- (589,114) ;
\draw    (414,106.5) -- (414,124.5) ;
\draw    (414,106.5) -- (420,106.5) ;
\draw    (414,124.5) -- (421,124.5) ;
\draw    (515,107.5) -- (515,127.5) ;
\draw    (515,107.5) -- (505,107.5) ;
\draw    (515,127.5) -- (505,127) ;
\draw    (237,103) -- (237,126) ;
\draw    (237,103) -- (244,103) ;
\draw    (237,126) -- (245,126) ;
\draw    (299,105) -- (299,126) ;
\draw    (299,105) -- (290,105) ;
\draw    (299,126) -- (291,126) ;
\draw    (117,114) ;
\draw  [color={rgb, 255:red, 54; green, 45; blue, 229 }  ,draw opacity=1 ][line width=3] [line join = round][line cap = round] (287,113) .. controls (287,113.33) and (287.33,114) .. (287,114) ;
\draw  [color={rgb, 255:red, 55; green, 46; blue, 229 }  ,draw opacity=1 ][line width=3] [line join = round][line cap = round] (245,115) .. controls (245,114.67) and (245,114.33) .. (245,114) ;
\draw  [color={rgb, 255:red, 54; green, 45; blue, 229 }  ,draw opacity=1 ][line width=3] [line join = round][line cap = round] (502,113) .. controls (502,113.33) and (502.33,114) .. (502,114) ;
\draw  [color={rgb, 255:red, 54; green, 45; blue, 229 }  ,draw opacity=1 ][line width=3] [line join = round][line cap = round] (488,113) .. controls (488,113.33) and (488.33,114) .. (488,114) ;
\draw  [color={rgb, 255:red, 54; green, 45; blue, 229 }  ,draw opacity=1 ][line width=3] [line join = round][line cap = round] (428,113) .. controls (428,113.33) and (428.33,114) .. (428,114) ;
\draw  [color={rgb, 255:red, 54; green, 45; blue, 229 }  ,draw opacity=1 ][line width=3] [line join = round][line cap = round] (442,113) .. controls (442,113.33) and (442.33,114) .. (442,114) ;
\draw  [color={rgb, 255:red, 54; green, 45; blue, 229 }  ,draw opacity=1 ][line width=3] [line join = round][line cap = round] (455,113) .. controls (455,113.33) and (455.33,114) .. (455,114) ;
\draw  [color={rgb, 255:red, 54; green, 45; blue, 229 }  ,draw opacity=1 ][line width=3] [line join = round][line cap = round] (472,113) .. controls (472,113.33) and (472.33,114) .. (472,114) ;
\draw    (294,153) -- (289.29,120.98) ;
\draw [shift={(289,119)}, rotate = 81.63] [color={rgb, 255:red, 0; green, 0; blue, 0 }  ][line width=0.75]    (10.93,-3.29) .. controls (6.95,-1.4) and (3.31,-0.3) .. (0,0) .. controls (3.31,0.3) and (6.95,1.4) .. (10.93,3.29)   ;
\draw    (403,154) -- (424.88,121.66) ;
\draw [shift={(426,120)}, rotate = 124.08] [color={rgb, 255:red, 0; green, 0; blue, 0 }  ][line width=0.75]    (10.93,-3.29) .. controls (6.95,-1.4) and (3.31,-0.3) .. (0,0) .. controls (3.31,0.3) and (6.95,1.4) .. (10.93,3.29)   ;
\draw    (514,150) -- (503.69,121.88) ;
\draw [shift={(503,120)}, rotate = 69.86] [color={rgb, 255:red, 0; green, 0; blue, 0 }  ][line width=0.75]    (10.93,-3.29) .. controls (6.95,-1.4) and (3.31,-0.3) .. (0,0) .. controls (3.31,0.3) and (6.95,1.4) .. (10.93,3.29)   ;
\draw    (222,162) -- (242.11,121.79) ;
\draw [shift={(243,120)}, rotate = 116.57] [color={rgb, 255:red, 0; green, 0; blue, 0 }  ][line width=0.75]    (10.93,-3.29) .. controls (6.95,-1.4) and (3.31,-0.3) .. (0,0) .. controls (3.31,0.3) and (6.95,1.4) .. (10.93,3.29)   ;

\draw (391,74) node [anchor=north west][inner sep=0.75pt]   [align=left] {$\displaystyle \delta u _{k}$};
\draw (507,73) node [anchor=north west][inner sep=0.75pt]   [align=left] {$\displaystyle v_{k}$};
\draw (200,70) node [anchor=north west][inner sep=0.75pt]   [align=left] {$ \displaystyle \delta u _{k+1}$};
\draw (293,69) node [anchor=north west][inner sep=0.75pt]   [align=left] {$\displaystyle v_{k+1}$};
\draw (103,123) node [anchor=north west][inner sep=0.75pt]   [align=left] {0};
\draw (496,153) node [anchor=north west][inner sep=0.75pt]   [align=left] {$\displaystyle a_{n_{k-1} +1} -t_{k}{}$};
\draw (361,156) node [anchor=north west][inner sep=0.75pt]   [align=left] {$\displaystyle a_{n_{k}} -t_{k}{}$};
\draw (251,158) node [anchor=north west][inner sep=0.75pt]   [align=left] {$\displaystyle a_{n_{k} +1} -t_{k+1}{}$};
\draw (155,160) node [anchor=north west][inner sep=0.75pt]   [align=left] {$\displaystyle a_{n_{k+1}} -t_{k+1}{}$};
\end{tikzpicture}
 \caption{An illustration of the intervals $I_k$ and $I_{k+1}$ and the points $b_n$ that we are going to choose. Inside each interval,  there are at most $N$ points. }
    \label{fig:a-t}
\end{figure}

 Clearly, $b_n\in E$ for $n\geq 1$ and $\lim_{n\to \infty} b_n=0$.  To see \eqref{e-e11}, by the definition of $(b_n)_{n=1}^\infty$, we see that for each $k\geq p+1$,
 \begin{equation}
 \label{e-e12}
 \frac{b_j-b_{j+1}}{a_j-a_{j+1}}=\left\{
 \begin{array}{ll}
  1, & \mbox{ if }n_{k-1} +1\leq j\leq n_{k}-1,\\
 \displaystyle \frac{a_{n_k}-a_{n_k+1} +t_{k+1}-t_k}{a_{n_k}-a_{n_k+1}}, & \mbox{ if }j= n_{k}.
 \end{array}
 \right.
 \end{equation}
 Recall that  $a_{n_k+1}<\delta a_{n_k}$ (see \eqref{eq:nk}), and  $0\leq t_k\leq (1-\delta)N^2\delta^{-N}\rho_k a_{n_k}$, so
 $$
\frac{|t_{k+1}-t_k|}{a_{n_k}-a_{n_k+1}}\leq \frac{t_k+t_{k+1}}{a_{n_k}-a_{n_k+1}}\leq (1-\delta)^{-1}\frac{t_k+t_{k+1}}{a_{n_k}}\leq N^2\delta^{-N}(\rho_k+\rho_{k+1}) \to 0
 $$
as $k\to \infty$. Combining this with \eqref{e-e12} yields \eqref{e-e11}.

Finally we show that the sequence $(b_n)_{n=1}^\infty$ is the image of $(a_n)_{n=1}^\infty$ under a bi-Lipschtiz map.  To see it, define a mapping $f:\R\to \R$ by
\begin{equation*}\label{eq:bilip}
	f(x)=\begin{cases}
		x, & \text{if $x\le0$}, \\
		\displaystyle\frac{a_n-x}{a_n-a_{n+1}}\cdot b_{n+1} + \frac{x-a_{n+1}}{a_n-a_{n+1}}\cdot b_n, & \text{if $a_{n+1}\le x\le a_n$}, \\
		x-a_1+b_1, & \text{if $x>a_1$}.
	\end{cases}
\end{equation*}
Clearly $f(a_n)=b_n$ for $n\geq 1$,  and $f$ is a continuous piecewise linear map, with slope $1$ on $(-\infty, 0]\bigcup (a_1,\infty)$,
and $(b_n-b_{n+1})/(a_n-a_{n+1})$ on $[a_{n+1}, a_n]$ for $n\geq 1$.
By \eqref{e-e11},  $f'(0)=1$, and all these slopes are positive, uniformly bounded away from zero, and from above by a constant, thus $f$ is bi-Lipschitz on $\R$.
\end{proof}
\medskip

\section{ The proofs of Theorems~\ref{p:nonemb} and \ref{t:iff}}
\label{S-3}

In this section we prove Theorems~\ref{p:nonemb} and \ref{t:iff}. We begin with the following.

\begin{lem}\label{lem-3.1}
	Let $(a_n)_{n=1}^\infty$ be a strictly decreasing sequence of positive numbers such that
	$$\lim_{k\to \infty}  a_n=0\quad\mbox{and} \quad \lim_{n\to \infty} \frac{a_{n+1}}{a_n}=1.$$
	 Then there exists a subsequence~$(a_{n_k})_{k=1}^\infty$  such that $\lim_{k\to \infty}a_{n_{k+1}}/a_{n_k}=1$ and
	\begin{equation}\label{eq:dis}
	a_{n_k}- a_{n_{k+1}}\leq 2 (a_{n_m}- a_{n_{m+1}})	\quad\mbox{ for all }\; k,m\in \N \mbox{ with }\; k> m.
	\end{equation}
	\end{lem}
\begin{proof}
Write for  $n\in \N$,
$$
t_n=\sup\{a_p-a_{p+1}:\; p\geq n\}.
$$
Clearly, $(t_n)$ is monotone decreasing. Since the sequence $(a_n)$ is strictly decreasing with limit $0$, the supremum in the above equality is attainable for each $n$; that is, for each $n$ there exists $p(n)\in \N$ such that $$p(n)\geq n\quad \mbox{ and }\quad t_n=a_{p(n)}-a_{p(n)+1}.$$

Next we inductively define a subsequence $(n_k)_{k=1}^\infty$ of natural numbers. Set $n_1=1$. Suppose $n_1,\ldots, n_k$ have been defined. Then we define
\begin{equation}
\label{e-e1}
n_{k+1}=\inf\{p\in \N\colon p>n_k,\;  a_{n_k}-a_p\geq t_{n_k}\}.
\end{equation}
Since $a_{n_k}-a_{p(n_k)+1}\geq a_{p(n_k)}-a_{p(n_k)+1}=t_{n_k}$,  it follows from \eqref{e-e1} that
$$n_k<n_{k+1}\leq p(n_k)+1<\infty.$$   Continuing this process, we obtain the sequence $(n_k)_{k=1}^\infty$.

By \eqref{e-e1}, $a_{n_k}-a_{n_{k+1}}\geq t_{n_k}$. Moreover, by \eqref{e-e1} and the definition of $t_{n_k}$,
$$
a_{n_k}-a_{n_{k+1}}= (a_{n_{k}}-a_{n_{k+1}-1})+(a_{n_{k+1}-1}-a_{n_{k+1}})\leq t_{n_{k}}+t_{n_{k}}=2t_{n_{k}}.
$$
That is,
\begin{equation}
\label{e-e2}
t_{n_{k}}\leq a_{n_{k}}-a_{n_{k+1}}\leq 2 t_{n_{k}}\quad  \mbox{  for all }k\geq 1.
\end{equation}
It follows that for any $k,m\in \N$ with $k>m$,
$$a_{n_{k}}-a_{n_{k+1}}\leq 2 t_{n_{k}}\leq 2 t_{n_m}\leq 2(a_{n_{m}}-a_{n_{m+1}}).$$

Finally we show that $\lim_{k\to \infty} a_{n_{k+1}}/a_{n_k}=1$, which is equivalent to  $$\lim_{k\to \infty} (a_{n_k}-a_{n_{k+1}})/a_{n_k}=0.$$
Notice that by \eqref{e-e2}, $a_{n_k}-a_{n_{k+1}}\leq 2t_{n_k}=2(a_{p(n_k)}-a_{p(n_k)+1})$,  and
$a_{n_k}\geq a_{p(n_k)}$. Hence
$$
0\leq \frac{a_{n_k}-a_{n_{k+1}}}{a_{n_k}}\leq \frac{2(a_{p(n_k)}-a_{p(n_k)+1})}{a_{p(n_k)}}=2\left(1-\frac{a_{p(n_k)+1}}{a_{p(n_k)}}\right)\to 0,
$$
as desired.
\end{proof}
\begin{prop}\label{prop-3.2}
	Let $(a_n)_{n=1}^\infty$ be a strictly decreasing sequence of positive numbers with $a_n\to0$ as $n\to\infty$. Suppose that
	\[ \sup_{m>n>1}\frac{a_{m-1}-a_m}{a_{n-1}-a_n}<\infty, \quad\mbox{ and }\quad  \limsup_{n\to\infty}\frac{a_{n+1}}{a_n}=1. \]
	Then there exists a  compact set $E\subset\R$ with positive Lebesgue measure such that $(a_n)_{n=1}^\infty$ can not be bi-Lipschitz embedded into~$E$.
\end{prop}
\begin{proof}
	By our assumption, there exists $C>1$ such that \begin{equation}\label{e-e3}
a_n-a_{n+1}\leq C (a_m-a_{m+1})\quad  \mbox{ for all } n,m \mbox{ with }n>m.
\end{equation}
Since $$\liminf_{n\to \infty} \frac{a_n-a_{n+1}}{a_n}=1-\limsup_{n\to \infty} \frac{a_{n+1}}{a_n}=0,$$   we can choose a  strictly increasing sequence $(n_k)_{k=1}^\infty$ of natural numbers such that
$$
\frac{a_{n_k}-a_{n_{k}+1}}{a_{n_k}}\leq k^{-2}4^{-k}\quad \mbox{ for }\;k\geq 1.
$$
For each $k\geq 1$, let $\ell_k$ be the smallest integer $\geq k/a_{n_k}$, and let $\delta_k=k(a_{n_k}-a_{n_k+1})$, Clearly,  it holds that
$$
\frac{1}{ \ell_k}\leq \frac{a_{n_k}}{k}<\frac{2}{\ell_k}
$$
and
\begin{equation}
\label{e-e4}
\ell_k\delta_k\leq \frac{2k}{a_{n_k}}\cdot k(a_{n_k}-a_{n_k+1})\leq  2\cdot 4^{-k}\quad \mbox{ for }k\geq 1.
\end{equation}

Define a sequence $(E_k)_{k=1}^\infty$ of compact subsets of $[0,1]$ by
$$
E_k=[0,1]\setminus \bigcup_{j=0}^{\ell_k}\left(\frac{j}{\ell_k}-\frac{\delta_k}{2},\; \frac{j}{\ell_k}+\frac{\delta_k}{2}\right).
$$
It is easy to see that for each $k$,  $E_k$ is the union of $\ell_k$ disjoint intervals of length $(1-\delta_k\ell_k)/\ell_k$, with a gap of length $\delta_k$ between any two adjacent intervals.

Set $E=\bigcap_{k=1}^\infty E_k$. Then $E$ is a compact set with Lebesgue measure
$$
\mathcal L (E)\geq 1-\sum_{k=1}^\infty\mathcal L\left([0,1]\setminus E_k]\right)\geq 1-\sum_{k=1}^\infty  \ell_k\delta_k\geq 1-\sum_{k=1}^\infty 2\cdot 4^{-k}=\frac13>0,
$$
where we have used \eqref{e-e4} in the third inequality.
Below we show by contradiction that $(a_n)_{n=1}^\infty$ can not be embedded into $E$ by a bi-Lipschitz map.

Suppose on the contrary that $(a_n)_{n=1}^\infty$ can be embedded into~$E$ by a bi-Lipschitz map~$f:\R\to \R$. Let $b_n=f(a_n)$ for $n\geq 1$ and $b_\infty=\lim_{n\to \infty} b_n$. Then $b_n,\; b_\infty\in E$.  Clearly $b_\infty=f(0)$,  and $(b_n)_{n=1}^\infty$ is strictly monotone increasing or monotone decreasing. Since $f$ is bi-Lipschitz,  there exists a constant $L>1$ such that
\begin{equation*}
\label{e-e5'}
L^{-1} \le \frac{|b_n-b_m|}{a_n-a_m}\le L \quad \mbox{ for all } n, m\in \N, n\neq m.
\end{equation*}
 In particular, this implies that
\begin{equation}\label{e-e5}
	L^{-1} \le \frac{|b_n-b_{n+1}|}{a_{n}-a_{n+1}} \le L	 \quad\text{and}\quad L^{-1} \le \frac{|b_n-b_\infty|}{a_n} \le L.
	\end{equation}
Now fix an integer $k>CL$.
 Then by (\ref{e-e5}) and \eqref{e-e3},  for all $m\ge n_k$,
\begin{equation}
\label{e-e6}
|b_m-b_{m+1}| \le  L (a_m-a_{m+1}) \leq CL (a_{n_k}-a_{n_k+1})<k (a_{n_k}-a_{n_k+1})=\delta_k.
\end{equation}
Meanwhile by \eqref{e-e5},
\begin{equation}
\label{e-e7}
|b_{n_k}-b_\infty|\geq \frac{a_{n_k}}{L}> \frac{a_{n_k}}{k}\geq \frac{1}{\ell_k}.
\end{equation}	
Notice that $(b_m)_{m=n_k}^\infty\subset E\subset E_k$. Recall that $E_k$ is the union of $\ell_k$ disjoint intervals of length $(1-\delta_k\ell_k)/\ell_k$, with a gap of length $\delta_k$ between any two adjacent intervals.  By \eqref{e-e6}, the sequence $(b_m)_{m=n_k}^\infty$ must be entirely contained in a component interval of $E_k$.    This forces that $|b_m-b_\infty|\leq (1-\delta_k\ell_k)/\ell_k$, which clearly contradicts \eqref{e-e7}.
\end{proof}
Now we ready to prove Theorems \ref{p:nonemb} and \ref{t:iff}.

\begin{proof}[Proof of Theorem~\ref{p:nonemb}]
It follows directly from Lemma \ref{lem-3.1} and Proposition \ref{prop-3.2}.
\end{proof}

\begin{proof}[Proof of Theorem~\ref{t:iff}]
The sufficiency part of the theorem follows from Theorem~\ref{t:BLE} and the necessity part follows from Proposition~\ref{prop-3.2}.
\end{proof}

\section{The proof of Theorem \ref{th:uniform}}\label{S-4}

We will prove Theorem \ref{th:uniform} in this section. To start, we will need to study whether we can bi-Lipschitz embed certain sequences with bi-Lipschitz bounds independent of the measurable sets. The following theorem provides a quantitative result in this direction, which may be of independent interest for other future study.

\begin{thm}\label{th:uniform-biLip}
	Let $A = (a_n)_{n=1}^{\infty}$ be a sequence of positive numbers such that
	\begin{equation}\label{eq_sum_condition}
		\delta:=a_1+\sum_{n=1}^{\infty} \frac{a_{n+1}}{a_n} <1/8.
	\end{equation}
	Then for any measurable set $E\subset [0,1]$ with $\mathcal{L}(E)>\frac{1}{2}+4\delta$, there exists a bi-Lipschitz map $f:\R\to\R$ such that $f(A)\subset E$ and
	\[ \frac{1}{2}|x-y|\le |f(x)-f(y)|\le \frac{3}{1-\delta} |x-y|. \]
\end{thm}

We remark that the requirement that the Lebesgue measure is uniformly bounded away from zero is necessary. Indeed, let $E = [0,\varepsilon_0]$. Clearly, any sequences $A = (a_n)_{n=1}^{\infty}$ converging to zero can be bi-Lipschitz embedded into $E$ via a map $f$. But then $|f(a_1)-f(0)|\le \varepsilon_0$ which means that the   bi-Lipschitz lower bound must be less than $\varepsilon_0/a_1$.  So there cannot be any uniform lower bound over all measurable sets of positive measures.

\medskip

The key to the proof of Theorem \ref{th:uniform-biLip} is the following lemma, which asserts that, for any measurable set  $E\subset [a,b]$ with sufficiently large density, there are two small subintervals of $[a,b]$ with pre-specified length and distance,  such that the restrictions of $E$ on these two intervals  satisfy certain density conditions.

\begin{lem}\label{lemma_Leb_estimate}
    Let $t\in(1/2,1)$ and let $0<\varepsilon<t-1/2$. Suppose that $M$ is an even integer such that $M>2/\varepsilon$. Then for  every interval $I = [a,b]$ and  every Lebesgue measurable set $E$ with ${\mathcal L}(E\cap I)\ge t\cdot {\mathcal L}(I)$, there exists $1\le j\le M-2$ such that
    \[ \frac{{\mathcal L}( E\cap I_j) }{{\mathcal L}(I_j)}\ge t-\varepsilon \quad\text{and}\quad   {\mathcal L}(E\cap I_{j+2})>0, \]
	where $I_j=\bigl[ a+\frac{j-1}{M}(b-a), a+\frac{j}{M}(b-a) \bigr]$.
\end{lem}

\begin{proof}
    Let $p_j = {\mathcal L( E\cap I_j) }/{{\mathcal L}(I_j)}$, then $p_j\in [0,1]$. Moreover,
    \[ \sum_{j=1}^M{\mathcal L}(E\cap I_j) = {\mathcal L}(E\cap I)\ge t\cdot {\mathcal L}(I). \]
	Splitting the sum into $j$ being odd or even, we must have
	$$
	\sum_{j=1}^{M/2} {\mathcal L}(E\cap I_{2j}) \ge \frac{t}{2}\cdot  {\mathcal L}(I)  \qquad \mbox{or}   \qquad \sum_{j=1}^{M/2} {\mathcal L}(E\cap I_{2j-1}) \ge \frac{t}{2}\cdot  {\mathcal L}(I).
	$$
	We may assume without loss of generality that the first case holds.  Since ${\mathcal L}(I_j) = \frac1{M}{\mathcal L}(I),$  it follows that
	\begin{equation}\label{eq_sum_p_j}
	  \sum_{j=1}^{M/2}p_{2j}  \ge \frac{M}{2}t.
	  \end{equation}
	Let $N = M/2$. We now claim that there exists $k\in \{1,2,..., N-1\}$ such that
	$$
	p_{2k}\ge t-\varepsilon \quad\text{and}\quad  \ p_{2k+2}>0.
	$$
	This will complete the proof. To justify the claim, we suppose on the contrary that the claim is false. Then for all $k\in \{1,...,N-1\},$ we have
$$
\mbox{ either }p_{2k}<t-\varepsilon, \quad \mbox{ or }\; p_{2k}\ge t-\varepsilon \mbox{ and }p_{2k+2} = 0.$$  We now define
	$$
	\Lambda_1 = \{k: p_{2k}\ge t-\varepsilon\}, \quad \Lambda_2 = \{k+1: k\in \Lambda_1\}.
	$$
	Then $\Lambda_1$ and~$\Lambda_2$ are disjoint subset of $\{1,\dots,N\}$.  Write also $\Lambda = \Lambda_1\cup\Lambda_2$.  Then $\#\Lambda_1 = \#\Lambda_2$ and hence $\#\Lambda = 2 \cdot \#\Lambda_1$. Here $\#$ denotes the cardinality of a set. Note that
	$$
	\begin{aligned}
	    \sum_{k=1}^Np_{2k} \le &  \left(\sum_{k\in \Lambda} p_{2k}\right) + \left(\sum_{k\in\{1,...,N-1\} \setminus\Lambda}p_{2k}\right)+p_{2N}\\
	    \le & \frac{1}{2} \cdot\#\Lambda + (t-\varepsilon) \cdot (N-\#\Lambda)+1\\
	    \le & (t-\varepsilon)N+1\\
< & tN,
	\end{aligned}
	$$
	 where in the last two inequalities, we have used the assumptions  $t-1/2>\varepsilon$ and $N = M/2>1/\varepsilon$.  But this contradicts (\ref{eq_sum_p_j}).
\end{proof}

We also need a preliminary lemma.
\begin{lem}\label{l:a2M}
	Let $(a_n)_{n=1}^\infty$ be a sequence of positive numbers such that
	\[ \delta:=a_1+\sum_{n=1}^{\infty}\frac{a_{n+1}}{a_n} \le\frac{1}{4}. \]
	Let
	\[ M_1=2\Bigl\lceil \frac{1}{2a_1} \Bigr\rceil \quad\text{and}\quad M_{n+1}=2\Bigl\lceil \frac{1}{2M_1\dotsm M_n a_{n+1}} \Bigr\rceil \]
	for all $n\ge1$, where $\lceil x \rceil$ denotes the smallest integer larger than $x$. Then
	\begin{equation}\label{eq:aM}
		\frac{a_n}{2}\le \frac{1}{M_1\dotsm M_n}\le a_n \quad\text{for all $n\ge1$},
	\end{equation}
	and
	\begin{equation}\label{eq:sumM}
		\sum_{n=1}^{\infty} \frac{1}{M_n}< 2\delta.
	\end{equation}
\end{lem}
\begin{proof}
	We first prove~\eqref{eq:aM} by induction on~$n$. By the definition of~$M_1$,
	\[ M_1\ge \frac{2}{2a_1}=\frac{1}{a_1}. \]
	On the other hand,
	\[ M_1\le 2\cdot\left( \frac1{2a_1}+1 \right) = \frac{1}{a_1}+2\le\frac{2}{a_1}, \]
	since $1/a_1\ge4$. Thus \eqref{eq:aM} holds for $n=1$.
	
	Now suppose this is true for $n=k$.  {\color{red} By} the definition of $M_{k+1}$,
	\[ M_1\dotsm M_k\cdot M_{k+1}\ge M_1\dotsm M_k\cdot \frac{2}{2M_1\dotsm M_k a_{k+1}}=\frac{1}{a_{k+1}}. \]
	On the other hand,
	\[ \begin{split}
		M_1\dotsm M_k\cdot M_{k+1} &\le M_1\dotsm M_k\cdot 2\cdot \bigl( (2M_1\dotsm M_k a_{k+1})^{-1} +1 \bigr) \\
		&= \frac{1}{a_{k+1}} + 2M_1\dotsm M_k \\
		&\le \frac{1}{a_{k+1}} + \frac{4}{a_k} \qquad(\text{by induction hypothesis}) \\
		&\le \frac{1}{a_{k+1}} + \frac{a_k}{a_{k+1}}\cdot\frac{1}{a_k} \qquad(\text{since } a_k/a_{k+1}\ge4 ) \\
& = \frac{2}{a_{k+1}}.
	\end{split} \]
 Hence, \eqref{eq:aM} holds for $n=k+1$. This completes the proof of \eqref{eq:aM}.
	
	Finally, we deduce~\eqref{eq:sumM} from~\eqref{eq:aM}. Indeed, by~\eqref{eq:aM},
	\[ \begin{split}
		\sum_{n=1}^{\infty}\frac{1}{M_n} &=\frac{1}{M_1} + \frac{M_1}{M_1M_2} +\dots+\frac{M_1\dotsm M_{n-1}}{M_1\dotsm M_n} +\dotsb \\
		&\le a_1 + \frac{2a_2}{a_1} +\dots+ \frac{2a_n}{a_{n-1}} + \dotsb\\
& < 2\delta. \qedhere
	\end{split} \]
\end{proof}

\begin{proof}[Proof of Theorem \ref{th:uniform-biLip}]
	Let $(a_n)_{n=1}^\infty$ be a sequence of positive numbers satisfying~\eqref{eq_sum_condition}. Let
	\[ M_1=2\Bigl\lceil \frac{1}{2a_1} \Bigr\rceil \quad\text{and}\quad M_{n+1}=2\Bigl\lceil \frac{1}{2M_1\dotsm M_n a_{n+1}} \Bigr\rceil \]
	for all $n\ge1$. By Lemma~\ref{l:a2M}, the sequence $(M_n)_{n=1}^\infty$ satisfies~\eqref{eq:aM} and~\eqref{eq:sumM}.

	For a given measurable set $E\subset[0,1]$ with ${\mathcal L}(E)>\frac{1}{2}+4\delta$, pick $\eta>0$ such that
	\[ {\mathcal L}(E)>\frac12+ 2(2+\eta)\delta. \]
	Define $\varepsilon_n = (2+\eta)/M_n$. Then $M_n>2/\varepsilon_n$. Let $t = 1/2+ (2+\eta)\cdot 2\delta$. By~\eqref{eq:sumM}, we have
	\begin{equation}\label{eq:t-eps}
		\sum_{i=1}^{\infty}\varepsilon_i< t- 1/2.
	\end{equation}
	
	Let $\Delta_0 = [0,1]$. We now construct inductively two sequence $\{\Delta_k\}_{k=1}^{\infty}$ and $\{\Delta_k'\}_{k=1}^{\infty}$ of intervals such that the following properties hold for all $k\ge 1$:
	\begin{enumerate}[(i)]
		\item \label{en:41i}$\Delta_{k-1}\supset \Delta_{k}\cup \Delta_{k}'$;
		\item \label{en:41ii}for some integer $1\le j_k\le M_1\dotsm M_k-2$,
		\[ \Delta_k=\frac{1}{M_1\dotsm M_k}\cdot[j_k-1,j_k] \quad\text{and}\quad \Delta_k'=\frac{1}{M_1\dotsm M_k}\cdot[j_k+1,j_k+2]; \]
		\item \label{en:41iii} $\displaystyle \frac{\mathcal{L}(E\cap\Delta_k)}{\mathcal{L}(\Delta_k)}\ge t-\sum_{i=1}^{k}\varepsilon_i$ and $\displaystyle \mathcal{L}(E\cap\Delta_k')>0$.
	\end{enumerate}
	
	To see this, we first apply Lemma \ref{lemma_Leb_estimate} on $\Delta_0$ to obtain $1\le j_1\le M_1-2$ such that
	\[ \frac{{\mathcal L}(E\cap\Delta_1)}{{\mathcal L}(\Delta_1)}\ge t-\varepsilon_1 \quad\text{and}\quad
	{\mathcal L}\left(E\cap \Delta_1'\right)>0, \]
	where $\Delta_1 =\left[\frac{j_1-1}{M_1}, \frac{j_1}{M_1}\right]$ and $\Delta_1' = \left[\frac{j_1+1}{M_1}, \frac{j_1+2}{M_1}\right]$. Clearly, all properties hold for $k=1$ with $\Delta_1$ and $\Delta_1'$.
	
	Suppose that for $k\ge 1$, we have chosen $\Delta_k$ and $\Delta_k'$ with all the properties hold.
	We subdivide $\Delta_k$ into $M_{k+1}$ intervals. By the fact $M_{k+1}\ge 2/\varepsilon_{k+1}$ and~\eqref{eq:t-eps}, we can apply Lemma \ref{lemma_Leb_estimate} to obtain $\Delta_{k+1}$ and $\Delta_{k+1}'$ inside $\Delta_k$ such that Property~\eqref{en:41iii} holds. Property~\eqref{en:41i} holds by our construction. Furthermore, Property~\eqref{en:41ii} also holds by our choice of $\Delta_{k+1}$ and $\Delta'_{k+1}$ in Lemma \ref{lemma_Leb_estimate}.
	
	\medskip
	
	By the second part of Property~\eqref{en:41iii},  we can pick $b_k\in E\cap \Delta_{k}'$ for each $k\ge 1$. Then $b_{k+1}\in \Delta_{k+1}'\subset \Delta_k$ by Property~\eqref{en:41i}. By Property~\eqref{en:41ii}, we have
	\[ \frac{1}{M_1\cdots M_k}\le b_{k}-b_{k+1}\le \frac{3}{M_1\cdots M_k}. \]
	Combining with~\eqref{eq:aM}, we have
	\[ \frac12 \le \frac{b_{k}-b_{k+1}}{a_k-a_{k+1}}\le \frac{3a_k}{a_k-\delta a_k} = \frac{3}{1-\delta}. \]
	By linearly interpolating all points between $(a_k,b_k)$, we obtain a bi-Lipschitz map $f$ with bi-Lipschitz constants $1/2$ and $3/(1-\delta)$ and $f(A)\subset E $. This completes the proof.
\end{proof}

\medskip


We are now ready to prove Theorem \ref{th:uniform} in the introduction.

\begin{proof}[Proof of Theorem~\ref{th:uniform}]

Let $A = (a_n)_{n=1}^{\infty}$ be a sequence of positive numbers such that
\[ \delta:=a_1+\sum_{n=1}^{\infty}\frac{a_{n+1}}{a_n}<\frac{1}{8}. \]
Let $E$ be a measurable set of positive measure. Replacing $E$ by a suitable translation of itself, we may assume that $0$ is a Lebesgue density point of $E$. By the Lebesgue density theorem, there exists $N\in\N$ such that for all $n>N$,
\[ \frac{{\mathcal L}(E\cap [3^{-n}, 2\cdot 3^{-n}])}{3^{-n}} > \frac{1}{2}+4\delta. \]

Let $g_n(x) = 3^nx-1$ and consider
\[ E_n = g_n (E\cap [3^{-n}, 2\cdot 3^{-n}]). \]
Then $E_n\subset [0,1]$ and ${\mathcal L}(E_n) >\frac{1}{2}+4\delta$ for all $n>N$. Using Theorem \ref{th:uniform-biLip}, for each $n>N$, we can find a bi-Lipschitz map $f_n\colon\R\to\R$ whose bi-Lipschitz lower and upper bounds are   $1/2$ and $3/(1-\delta)$ respectively and $f_n (A)\subset E_n$. Let $h_n = g_n^{-1}\circ f_n \circ g_n$. A direct check shows that $h_n$ is also bi-Lipschitz on $\R$ with bi-Lipschitz lower and upper bounds $1/2$ and $3/(1-\delta)$ respectively. Moreover,
\[ h_n \left(3^{-n} (1+A)\right)\subset E\cap [3^{-n}, 2\cdot 3^{-n}]. \]
Define the map $h:\R\to\R$ by
\[ h(x)=h_n(x) \quad\text{if $x\in [3^{-n},3^{-n}(1+a_1)]$ for some $n>N$}. \]
Then we extend $h$ continuously by a linear function on  each  of the intervals in the complement, i.e., $h$ is a straight line  on each  interval $[3^{-n-1}(1+a_1), 3^{-n}]$  with $n>N$ and on the unbounded intervals, we simply define it with a straight line of slope~$1$.

\medskip

We first claim that $h$ is a bi-Lipschitz function. Indeed, $h_n$ are all bi-Lipschitz with uniform bounds $1/2$ and $3/(1-\delta)$. It suffices to show that the  slopes of $h$ on the intervals $[3^{-n-1}(1+a_1), 3^{-n}]$ are uniformly bounded. To see this, because $h_n (3^{-n}(1+A))\subset E\cap [3^{-n},2\cdot 3^{-n}]$,
\[ h_n(3^{-n})\in [3^{-n}, 2\cdot3^{-n}], \quad  h_n(3^{-n-1}(1+a_1))\in  [3^{-n-1}, 2\cdot3^{-n-1}]. \]
the  slope of $h$ on $[3^{-n-1}(1+a_1), 3^{-n}]$ is therefore lying in the interval
\[ \left[\frac{3^{-n}-2\cdot 3^{-n-1}}{3^{-n-1}(2-a_1)},\; \frac{2\cdot3^{-n}- 3^{-n-1}}{3^{-n-1}(2-a_1)}\right] =\left[\frac{1}{2-a_1}, \; \frac{5}{2-a_1}\right].  \]
Hence, $h$ is a bi-Lipschitz function with bi-Lipschitz lower bound
\[ \min \Bigl( \frac{1}{2}, \frac{1}{2-a_1} \Bigr)=\frac{1}{2} \]
and upper bound
\[ \max\Bigl( \frac{3}{1-\delta}, \frac{5}{2-a_1} \Bigr)=\frac{3}{1-\delta}, \]
since $0<a_1<\delta<1/8$.

\medskip

\begin{figure}[h]

\begin{tikzpicture}[x=0.75pt,y=0.75pt,yscale=-1,xscale=1]

\draw    (219,1) -- (220,292) ;
\draw    (199,270) -- (586,269) ;
\draw    (470,259) -- (470,279) ;
\draw    (470,259) -- (479,259) ;
\draw    (517,280) -- (526,280) ;
\draw    (477,279) -- (470,279) ;
\draw    (517,261) -- (525,261) ;
\draw    (525,261) -- (525,280) ;
\draw    (342,260) -- (352,260) ;
\draw    (351,279) -- (342,279) ;
\draw    (342,260) -- (342,279) ;
\draw    (361,260) -- (371,260) ;
\draw    (371,280) -- (364,280) ;
\draw    (371,260) -- (371,280) ;
\draw [color={rgb, 255:red, 227; green, 17; blue, 17 }  ,draw opacity=1 ]   (471,155) -- (368,229) ;
\draw    (368,229) -- (357,242) ;
\draw    (357,242) -- (347,247) ;
\draw    (476,151) -- (484,148) ;
\draw    (476,151) -- (471,155) ;
\draw    (484,148) -- (489,138) ;
\draw    (212,252) -- (228,252) ;
\draw    (211,214) -- (227,214) ;
\draw    (211,160) -- (230,160) ;
\draw    (212,127) -- (228,127) ;
\draw  [dash pattern={on 0.84pt off 2.51pt}]  (198,160) -- (540,159) ;
\draw  [dash pattern={on 0.84pt off 2.51pt}]  (212,127) -- (543,127) ;
\draw  [dash pattern={on 0.84pt off 2.51pt}]  (185,252) -- (527,252) ;
\draw  [dash pattern={on 0.84pt off 2.51pt}]  (205,215) -- (532,213) ;

\draw (462,280.4) node [anchor=north west][inner sep=0.75pt]    {$3^{-n}$};
\draw (502,280.4) node [anchor=north west][inner sep=0.75pt]    {$3^{-n}( 1+a_{1})$};
\draw (353,280.4) node [anchor=north west][inner sep=0.75pt]    {$3^{-n-1}( 1+a_{1})$};
\draw (304,280.4) node [anchor=north west][inner sep=0.75pt]    {$3^{-n-1}$};
\draw (265,54.4) node [anchor=north west][inner sep=0.75pt]    {$s=\ \dfrac{h_{n}( 3^{-n}) -h_{n+1}( 3^{-n-1}( 1+a_{1}))}{3^{-n} -3^{-n-1}( 1+a_{1})}$};
\draw (164,242.4) node [anchor=north west][inner sep=0.75pt]    {$3^{-n-1}$};
\draw (152,205.4) node [anchor=north west][inner sep=0.75pt]    {$2\cdot 3^{-n-1}$};
\draw (165,146.4) node [anchor=north west][inner sep=0.75pt]    {$3^{-n}$};
\draw (152,114.4) node [anchor=north west][inner sep=0.75pt]    {$2\cdot 3^{-n}$};
\draw (221,4.4) node [anchor=north west][inner sep=0.75pt]    {$\ h( x)$};
\draw (590,265.4) node [anchor=north west][inner sep=0.75pt]    {$x$};

\end{tikzpicture}

 \caption{An illustration of $h(x)$ for $3^{-n-1}\le x\le 3^{-n}(1+a_1)$ with slope in $[3^{-n-1}(1+a_1), 3^{-n}]$ equal to $s$. }
    \label{fig:slop}
\end{figure}
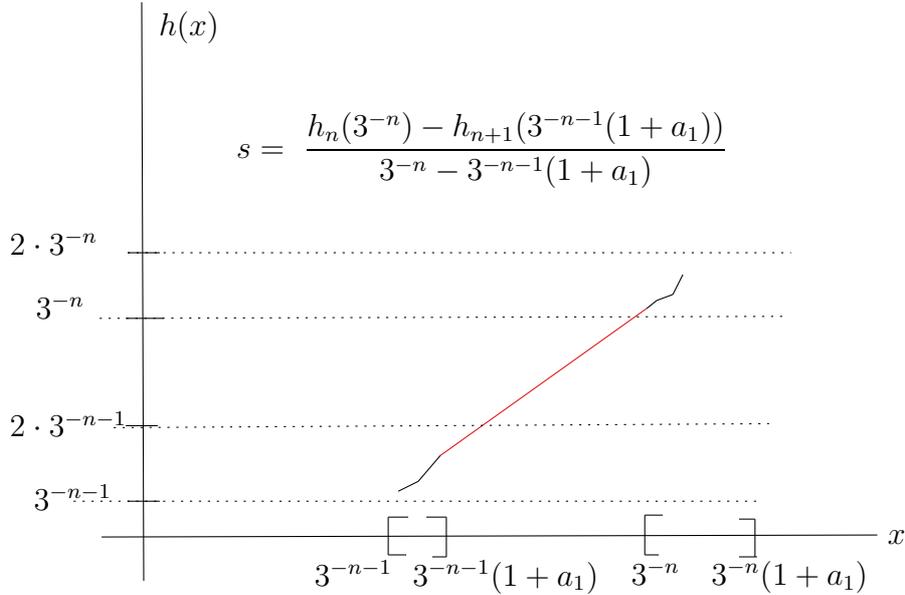

By our construction of $h$, $h \left(\bigcup_{n>N}3^{-n}(1+A)\right)\subset E$. We now consider the map $H(x) = h(3^{-N}x)$. Clearly, $H$ is still a bi-Lipschitz map and
\[
H \biggl( \bigcup_{n=1}^{\infty}3^{-n}(1+A) \biggr) = h\biggl( \bigcup_{n>N}3^{-n}(1+A) \biggr) \subset E.
\]
This completes the proof.
\end{proof}

{\noindent \bf Acknowledgements}.
 The authors would like to thank Xiong Jin  for proposing the problem addressed in this paper and some helpful discussions.  Feng was partially supported by the General Research Fund grant (project CUHK14305722) from the
Hong Kong Research Grant Council, and by a direct grant for research from the Chinese University
of Hong Kong. Xiong
was partially supported by NSFC grant 12271175 and 11871227.

\end{document}